\documentclass[11pt]{article}

\usepackage{fullpage, amsmath, amssymb, amsthm, graphicx}
\title{Constructive Method for Finding the Coefficients of a Divided Symmetrization}

\author{Nate Ince }
\date{}

\newtheorem{theorem}{Theorem}[section]
\newtheorem*{theorem*}{Theorem}

\newtheorem{proposition}[theorem]{Proposition}
\newtheorem{lemma}[theorem]{Lemma}
\newtheorem{corollary}[theorem]{Corollary}

\begin{document}
\maketitle
\begin{abstract}
We consider a type of divided symmetrization  $\overrightarrow{D}_{\lambda,G}$ where $\lambda$ is a nonincreasing partition on $n$ and where $G$ is a graph. We discover that in the case where $\lambda$ is a hook shape partition with first part equal to 2, we may determine the expansion of $\overrightarrow{D}_{\lambda,G}$ over the basis of Schur functions. We show a  combinatorial construction for finding the terms of the expansion and a second construction that allows computation of the coefficients. 
\end{abstract}

\section{Introduction}

Given a function $f$ on the variables $x_1, \dots, x_n$, the divided symmetrization of $f$ is \[ DS(f)=\sum_{\delta \in S_n} \frac{f(x_{\delta(1)}, \dots, x_{\delta(n)})}{\prod_{(i,j) \in E_n} (x_{\delta(i)}-x_{\delta(j)}) } \]

where $S_n$ is the symmetric group on $n$ and $E_n = \{ (i, j): i<j, i,j \in [n]\}$. 

A common example of a divided symmetrization is the function \[D_{\lambda, G} = \sum_{\delta \in S_n} \frac{x_{\delta(1)}^{\lambda_1} \dots x_{\delta(n)}^{\lambda_n} }{\prod_{(i,j) \in E(G)} (x_{\delta(i)}-x_{\delta(j)}) } \] where $G$ is a graph on $[n]$ whose edges are pairs $(i,j)$ with $i < j$. If $\lambda_1 + \cdots + \lambda_n =n-1$, $\lambda_1 \geq \cdots \geq \lambda_n \geq 0$, $P$ is the path graph with edges $(1,2), (2,3), \dots, (n-1,n)$, and $T$ is any tree on $[n]$, then $D_{\lambda, P}$ and $D_{\lambda, T}$ evaluate to constants. Alexander Postnikov finds a combinatorial interpretation to the evaluation of the constant that $D_{\lambda, P}$ evaluates to in \cite{postbook}, Theorem 3.2, and also connects $D_{\lambda, P}$ to polytope theory. Additionally, Petrov finds a combinatorial interpretation of the constant that $D_{\lambda, T}$ evaluates to in \cite{petbook}. Other divided symmetrizations are studied and evaluated in \cite{compute} 

The main consideration of this paper will be the divided symmetrizations $\overrightarrow{D}_{\lambda, G} $, which are defined as $\overrightarrow{D}_{\lambda, G} = D_{\lambda+o^{E(G)}, G}$ for $\lambda_1 + \cdots + \lambda_n =n$, $\lambda_1 \geq \cdots \geq \lambda_n$, for $G$ with $E(P) \subseteq E(G)$, and with $o^E$ defined as the  integer vector with $o^E_i = |\{(i,j) \in E\}|$. For instance if $G$ is the graph that contains edges $(1,2)$, $(2,3)$, $(3,4)$, $(4,5)$, and $(2,5)$, then $o^{E(G)}=(1,2,1,1,0)$.  These divided symmetrizations are a form of the ones discussed in \cite{postbook} Chapter 4. 

$\overrightarrow{D}_{\lambda, G}$ does not evaluate to a constant most of the time. However, we can still aim to find the coefficients of the divided symmetrization's expansion into symmetric monomials or over some other basis for symmetric polynomials, and we can seek a combinatorial formula for these coefficients. 

We say $\lambda \vdash n$ if $\lambda_1 + \cdots + \lambda_n =n$ and $\lambda_1 \geq \cdots \geq \lambda_n$. Represent $x_{\delta(1)}^{y_1} \cdots x_{\delta(n)}^{y_n}$ as $x_\delta^y$ and define $K_{\lambda,\mu}$ as the Kostka numbers and $m_y(x)$ as the symmetric monomial $\sum_{\delta \in S_n} x_\delta^y$. In \cite{postbook} Theorem 4.3, it is found that $\overrightarrow{D}_{\lambda, P}= \sum_{\substack{\mu \vdash n \\ K_{\lambda, \mu} \neq 0}} m_\mu(x) $, but no other cases are solved.

To classify $\overrightarrow{D}_{\lambda, G}$ in general, we first may rewrite the symmetric division form of $\overrightarrow{D}_{\lambda,G}(x)$ as follows:

\begin{align*}
\overrightarrow{D}_{\lambda,G}(x) &= \sum_{\delta\in S_n} \frac{ x_\delta^{\lambda+o^G}}{\prod_{(i,j) \in E(G)} (x_{\sigma(i)} - x_{\sigma(j)} ) } \\
&= \sum_{\delta\in S_n} \frac{ x_\delta^{\lambda+o^G}\prod_{(i,j) \in E_n- E(G)} (x_{\delta(i)}-x_{\delta(j)})}{\prod_{(i,j) \in E_n} (x_{\sigma(i)} - x_{\sigma(j)} )}  \\
&= \sum_{\delta\in S_n} \frac{(-1)^{\text{sign}(\delta)} x_\delta^{\lambda+o^G}\prod_{(i,j) \in E_n- E(G)} (x_{\delta(i)}-x_{\delta(j)})}{\prod_{(i,j) \in E_n} (x_{i} - x_{j} )}  \\
\end{align*}

For any $t \in \mathbb{Z}^n$ with all $t_i \geq 0$, let $l(t) \in \mathbb{Z}^n$ be defined as $l_i(t) = t_{\mu(i)} - (n-i)$ where $\mu \in S_n$ so that $t_{\mu(1)} > t_{\mu(2)} > \cdots > t_{\mu(n)}$.  If $x^t$ appears in the expansion of $x_\delta^{\lambda+o^{E(G)}}\prod_{(i,j) \in E_n- E(G)} (x_{\delta(i)}-x_{\delta(j)})$, then 
\[
\sum_{\delta \in S_n} \frac{(-1)^{\text{sign}(\delta)} x_\delta^t}{\prod_{i,j \in [n], i<j} (x_{i} - x_{j})} 
\]
is a term in the expansion of  $\overrightarrow{D}_{\lambda,G}(x)$. This expression evaluates to $0$ if $t_i=t_j$ for some $i,j \in [n]$ and otherwise evaluates the Schur function $s_{l(t)}(x)$.

Therefore if $\mathcal{N}_{\lambda, E}$  is the set of $t \in \mathbb{Z}^n_{\geq 0}$ so that there is a term $c_t x^t$, $c_t \neq 0$ in the expansion of $x_\delta^{\lambda+o^E}\prod_{(i,j) \in E_n- E} (x_{\delta(i)}-x_{\delta(j)})$ and so that there is no $i,j \in [n], i \neq j$ so that $t_i=t_j$, then:
 
\begin{align}
\overrightarrow{D}_{\lambda,G}(x) &= \sum_{t \in \mathcal{N}_{\lambda, E_n-E(G)}} \left( \sum_{l(t)=\lambda} c_t \right) s_{\lambda} \label{eqcoeff}
\end{align}

Therefore if we can determine the set $\mathcal{N}_{\lambda, E_n-E(G)}$ and find a combinatorial formula for the nonzero coefficients $c_t$ of the expansion, we may classify the evaluation of $\overrightarrow{D}_{\lambda,G}$ in a similar manner to the other divided symmetrizations mentioned earlier. 

Note that $\mathcal{N}_{\lambda, E_n-E(G)} \subseteq \mathcal{N}_{\lambda, E_n}$. Despite 
$\mathcal{N}_{\lambda, E_n}$ not really corresponding to the terms of one of our divided symmetrizations, finding its contents will be important to our approach. Section 2 will allow us to view the contents of $\mathcal{N}_{\lambda, E_n}$ combinatorially and will show why we concern ourselves specifically with the case $\lambda= (2,1,1,\dots,1,0)$. In our main results, we will find our classification of the coefficients for $\overrightarrow{D}_{\lambda,G}$ where $\lambda$ is a 2-hook and $G$ is any path with edges. For that purpose, we first find a construction for $\mathcal{N}_{\lambda, E_n}$ for $\lambda$ a 2-hook in Section 3. Then in Section 4 we find a construction for a set that allows us a combinatorial interpretation of the coefficients $c_t$ of $\overrightarrow{D}_{\lambda,G}$.  Section 5 considers the enumeration of $|\mathcal{N}_{\lambda, E_n}|$.

\section{Combinatorial Structure of $\mathcal{N}_{\lambda, G}$}

In order to have proper language to classify when an integer $n$-tuple $t$ is a member of $\mathcal{N}_{\lambda, E_n-E(G)}$, we redefine the problem as follows. Let $\omega^n = (n-1, n-2, \dots, 0)$. For any $e=(a,b) \in E_n$, let $v(e)$ be the vector $w \in \mathbb{Z}^n$, where \[w_i=\begin{cases}-1 &(i=1) \\ 1 &(i=b) \\ 0 &(i \neq a,b) \end{cases} .\] We say $w$ \emph{starts at index $a$} and \emph{ends at index $b$}.

For $E \subseteq E_n$, let $v(E)= \sum_{e \in E} v(e)$. A tuple $t$ then only appear in the expansion of\\ $x^{\lambda+o^{E(G)}}\prod_{(i,j) \in E_n- E(G)} (x_{i}-x_{j})$ if and only if there is some $E \subseteq E_n-E(G)$ so that $\lambda+\omega^n+v(E')=t$, as including the pair $(i,j) \in E$ corresponds to taking $-x_j$ from the term $(x_i-x_j)$ when expanding the product $\prod_{(i,j) \in E_n- E(G)} (x_{i}-x_{j})$.  The sign the term will have is $(-1)^{|E|}$, which gives us a description for $c_t$

\begin{lemma}\label{lctformula}
For any $\lambda \vdash n$ and for any $G$ a path on $n$ with edges, the coefficients $c_t$ of \ref{eqcoeff} have the form:
\[
c_t = \sum_{\substack{E\subseteq E_n-E(G) \\ \lambda+\omega^n+v(E)=t}} (-1)^{|E|}
\] 
\end{lemma}
Therefore, let us call $E \subseteq E_n-E(G)$ a \emph{justifying pair set} for $t$ if $\lambda+\omega^n+v(E)=t$. An integer $n$-tuple $t$ is in $\mathcal{N}_{\lambda,E_n-E(G)}$ 
if and only if at least one justifying pair set exists for $t$ and $t$ has all distinct terms.

To find all $t \in \mathcal{N}_{\lambda,E_n-E(G)}$, we will examine the possible integer $n$-tuples $l(t)$ can be. One might wonder if it might be possible to have $t \in \mathcal{N}_{\lambda,E_n}$ with $l(t)=\lambda$, justified by some nonempty $E \subseteq E_n$. Our first minor result will be to show that it is not. 

\begin{lemma}\label{selflemma}
For any $n$-tuple $x$ with $x_1 > x_2 > \cdots > x_n > 0$, if $x + v(E)=x_\delta$ for some $\delta \in S_n$, then $\delta$ is the identity permutation and $E=\emptyset$.
\end{lemma}

The proof of this lemma will work on a principle that the difference between $x_k+v_k(E)$ and $x_k$ can be at most $k-2$ because at most $k-2$ pairs in $E_n$ end at index $k$, which is an argument we will use repeatedly throughout this paper.

\begin{proof}

We go by induction on $n$. For $n=1$, the lemma's conditions hold trivially. Assume the lemma holds for $n-1$-tuples, and that there is an $n$-tuple $x$ with $x_1 > x_2 > \cdots > x_n > 0$, where there is some $E$ so that $x + v(E)=x_\delta$. 

Suppose that $\delta(k)=1$. Consider that $x_1-x_k > k-1$, while only $k-2$ pairs in $E_n$ exist that end at index $k$. Therefore this is not a possibility. If instead $\delta(1)=1$, then let $x'=(x_2, \dots, x_n)$. Let $\delta' \in S_{n-1}$ be defined as $\delta'(i) = \delta(i+1)$ for $i \in [n-1]$, and let $E'=\{(i,j)| (i+1,j+1) \in E, i,j \in [n-1]\}$. Then $x' + v(E')=x'_{\delta'}$. Applying inductive hypothesis shows $\delta'$ is the identity of $S_{n-1}$, which shows $\delta$ itself is the identity of $S_n$. Therefore the induction is complete. 
\end{proof}

Recall that the partition $\lambda$ is an \emph{$x$-hook shape} or a \emph{hook shape of length $n-x$} if $\lambda_1=x$ and $\lambda_i=1$ for $i \in \{2,\dots,n-x+1\}$, and $\lambda_i=0$ for $i \in \{n-x+2,n\}$. As it turns out, 2-hooks have special properties that allow us to determine everything in the expansion of $\overrightarrow{D}_{\lambda,G}$ that no other shape allows. Even though we are viewing partitions as nonnegative $n$-tuples, we refer to the number of nonzero parts as the \emph{length} of the partition. 

Using a similar idea to the last proof, we can limit our search for $t \in \mathcal{N}_{\lambda, E_n}$ to when $\lambda$ is a hook shape.

\begin{lemma}\label{hooklemma}
If $\lambda$ is a hook-shaped partition and $t\in \mathcal{N}_{\lambda, E_n}$, $t \neq \lambda+\omega^n$, then $l(t)$ must be a hook-shaped partition at least as long as $\lambda$. Furthermore, $t_1>t_i$ for $i \in \{2,\dots,n\}$. 
\end{lemma}

\begin{proof}Assume that $\lambda$ is any hook shape. For every index $i \in \{2, n\}$, $\lambda_i+\omega^n_i$  has value $n+1-i$ and can only take from at most $i-2$ previous values, so for any pair set $E$, $(\lambda_i+\omega^n_i+v_i(E)) \geq n-1$. Since $\lambda_i+\omega^n_i+v_i(E)$ must take on $n-1$ different values for $i \in \{2,n\}$, $\lambda_i+\omega^n_i+v_i(E)=n-1$ for some index $i$. Therefore for any pair set $E$, if $\lambda+\omega^n+v(E) \in \mathcal{N}_{\lambda, E_n}$ then $l(\lambda_2+\omega^n_2+v_2(E))=1$, which means $l(\lambda+\omega^n+v(E))$ must be a hook shape. There are no pairs that end at index 1, so $\lambda_1+\omega_1^n \leq \lambda_1+\omega_1^n + v_1(E)$ for any $E \subseteq E_n$, and therefore $l(\lambda+\omega^n+v(E))$ must be a hook shape at least as long as $\lambda$. 
\end{proof}

Let $\lambda^n$ be the hook-shaped partition $(2,1,1,\dots,1,0) \in \mathbb{Z}^n$, which we call the 2-hook of length $n$. According to \ref{hooklemma}, if $t \in \mathcal{N}_{\lambda, E_n}$ then $l(t)=(1,1,\dots,1)$ or $(2,1,\dots,1,0)$, and in the latter case $t=\lambda + \omega^n$ by \ref{selflemma}. We also note that $\lambda_i+\omega^n_i=n+1-i$ for $1<i<n$, and there are $i-2$ pairs in $E_n$ that end at index $i$ and $n-1-i$ pairs in $E_n$ that begin at index $i$. These facts mean that $2 \leq t_i+v_i(E) \leq n-1$ for any $E \subseteq E_n$ and for $1<i<n$, giving us the following corollary. 

\begin{corollary}\label{l2hform}
If $t\in \mathcal{N}_{\lambda^n, E_n-E(G)}$, then either $t=\lambda^n$ or $t_1, \dots, t_n$ are a permutation of $[n]$ with  $t_1=n$ and $t_n=1$. 
\end{corollary}

Since this shows that for 2-hooks the elements of $\mathcal{N}_{\lambda^n,E_n}$ can be viewed as permutations, we call the elements of $\mathcal{N}_{\lambda^n,E_n}$ \emph{2-hook permutations}. We now can describe how to construct $\mathcal{N}_{\lambda^n,E_n}$.

\section{The 2-hook Construction}

We call an integer $n$-tuple an \emph{initial construction} if it has the form $(n,0,\dots,0)$, and say the integer $n$-tuple $p'$ was \emph{obtained from a placement  of value $h$ on index $k$ from $p$} if $p'_i=p_i$ for $i \in [n]-\{k\}$ and $p'_k=h$. 

The \emph{2-hook construction} is the process that starts with an initial construction $p^{(1)}$ and obtains each $p^{(i)}$ from $p^{(i-1)}$ from a placement of value $n-i+1$ on $p^{(i-1)}$ for an index $k \in [n]$ that is \emph{legal for placement}, ending on $p^{(n)}$. Since the value is known each stage, we may drop it from the terminology as needed. Legality of placement will be described soon. An integer $n$-tuple is a \emph{partial construction} if it can be obtained as some $p^{(k)}$ of a 2-hook construction.

For any integer vector $v \in \mathbb{Z}^n$, let $I(v)=\{i \in [n], v_i \neq 0\}$. If for a partial construction $p$ we have $I(p) = \{a, a+1, \dots, b-1, b\}$ for some $a, b \in [n]$, $1 \leq a \leq b \leq n$, then we say $p$ is \emph{unbroken}. Otherwise, $p$ is \emph{broken}.

We say index $i$ is \emph{legal for placement on $p$} if $i \in [n]-I(p)$ and the following rules hold:

\begin{itemize}
\item If $I(p) \neq [n-1]$:
\begin{itemize}
\item If $p$ is unbroken, we may place on any $i \in [n-1]-I(p)$ 
\item If $p$ is broken we may place only on an $i \in [n-1]-I(p)$ if either $i+1 \in I(p)$ or $i-1 \in I(p)$. 
\end{itemize}
\item If $I(p)=[n-1]$ we place on index $n$.
\end{itemize}

Let's show an example of the 2-hook construction. The initial construction for $n=7$ is the $7$-tuple $p^{(0)}=(7,0,0,0,0,0,0)$. Since $I(p^{(0)})=\{1\}$, $p^{(0)}$ is unbroken. To obtain $p^{(1)}$ we may make a placement of value $6$ on any of the indices $\{2,3,4,5,6,7\}$. Suppose we obtain the partial construction $p^{(1)}=(7,0,0,0,6,0,0)$ by  placement on index $5$. Then $I(p^{(1)}) = \{1,5\}$ and $p^{(1)}$ is broken, so $p^{(2)}$ must be obtained by placing the value $5$ on one of the indices $\{2,4,6\}$. 

Here is how this 2-hook construction could go from start to finish:
\begin{align}
p^{(0)}&=(7,0,0,0,0,0,0)\\
p^{(1)}&=(7,0,0,0,6,0,0)\\
p^{(2)}&=(7,0,0,5,6,0,0)\\
p^{(3)}&=(7,4,0,5,6,0,0)\\
p^{(4)}&=(7,4,0,5,6,3,0)\\
p^{(5)}&=(7,4,2,5,6,3,0)\\
p^{(6)}&=(7,4,2,5,6,3,0)\\
p^{(7)}&=(7,4,2,5,6,3,1) \label{example}
\end{align}

It can be checked that $p^{(7)}=\lambda^7+\omega^7+v(E)$, for

\begin{align*}
E&=\{(1,5),(2,5),(3,5), (2,5), (2,4), (3,6), (3,7)\}
\end{align*}

Note in our example that $I(p^{(k)})$ is always two runs of consecutive numbers when $p^{(k)}$ is broken, specifically for $k \in \{1,2,3,4\}$. Then $p^{(5)}$ becomes unbroken from a placement that joins the two runs of $I(p^{(4)})$ together. This can be seen to always be the case from examining the placement rules. Therefore we can provide another way of looking at the placement rules that sheds some light on the structure of a partial construction.

Given a partial construction $p$, let $h(p)$ be the next value to be placed on $p$. Let $x_i(p)$, $x_o(p)$, and $y_i(p)$ be defined as follows: For broken $p$, let $x_i(p)=x_i$, $x_o(p)=x_o$, $y_i(p)=y_i$ so that $I(p)=\{1,2,\dots,x_i-1\} \cup \{x_o+1,x_o+2,\dots,y_i-1\}$. For $p$ unbroken, let $x_i(p)=x_i$ so that $I(p)=\{1,2,\dots,x_i-1\}$ and let  $x_o(p)=0$ and $y_i(p)=0$. Then when $p$ is unbroken, only placement on $x_i(p)$ results in a new unbroken construction. When $p$ is broken, only placements on $x_i(p)$, $x_o(p)$, and $y_i(p)$ are allowed. An unbroken state may be obtained from broken partial construction $p$ if $x_i(p) = x_o(p)$ and a placement is made on $x_i(p)$.
 
We assert that the 2-hook construction describes $\mathcal{N}_{\lambda^n,E_n}$.

\begin{proposition}\label{mainprop}
Every 2-hook permutation of length $n$ is obtained from the 2-hook construction, and every outcome of the 2-hook construction is a 2-hook permutation.
\end{proposition}

Since we cannot verify a 2-hook permutation without a justifying pair set, we will describe the \emph{strong 2-hook construction}, which is a similar construction to the weaker 2-hook construction but builds a pair set alongside each partial construction. Each step of the strong 2-hook construction will consist of a partial construction and an associated pair set, which we will call a \emph{state}. 

More formally, a \emph{state on $n$}, $S$, is defined as consisting of $S = (p^S, E^S)$ for $p^S$ is a partial construction and $E^S \subseteq E_n$ that occurs during the strong 2-hook construction.  We additionally define $h^S = h(p^S)$, $x_i^S=x_i(p^S)$, $x_o^S=x_o(p^S)$, and $y_i^S=y_i(p^S)$. The initial state $S_0$ on $n$ will consist of $p^{S^{(0)}}$ as the initial construction on $n$ and of $E^{S{(0)}}=\emptyset$. 

The strong 2-hook construction will work like the  2-hook construction in that we are creating objects $S^{(0)}, S^{(1)}, \dots, S^{(n)}$ where the following hold: $S^{(0)}$ is the initial state on $n$ and $S^{(i)}$ is obtained from $S^{(i-1)}$, if possible, by \emph{a placement on index $k$ on $S^{(i-1)}$}, which we will define shortly. 

For state $S$, let $s^S=\lambda^n+\omega^n + v(E^S)$. The intention of the rules of placement for the strong 2-hook construction is that for each $k \in \{0,\dots,n-1\}$, we have $p_i^{S^{(k)}} \leq s_i^{S^{(k)}}$ for all $i \in [n]$ and we have $p^{S^{(n)}}=s^{S^{(n)}}$. The rules for the position of index $i$ that we can place the value $h^S$ on for state $S$ are the same as in the weak 2-hook construction except that there exist cases where the placement might not succeed. If the placement is successful, we may obtain a new state $S'$ with $p^{S'}$ obtained from the placement on index $k$ from $p^S$. 

We define $E^S_k = \{1,2,\dots,k-2\} \cap I(s^S-p^S)$ and $M^S_k= \lambda^n_k+\omega^n_k+v_k(E^S \cup E^S_k)$. If $k$ is a legal placement for $p^S$ in the weak 2-hook construction and $h^S \leq M^S_k$, then consider the partial construction $p$ obtained by placing $h$ on index $k$ in $p^S$ and $E=E^S \cup E^S_k$, and $s=\lambda^n+\omega^n+v(E)$. By the definitions we have chosen, $p_i \leq s_i$ for $i \in [n]$, so $(p,E)$ would satisfy our expectations if it were a state. Similarly, if $h^S \leq s_k^S$, then the hypothetical state $S'$ obtained from $S$ by placing $h^S$ on index $k$ and letting $E^{S'}=E^S \cup E^S_k$ would have $p_i^{S'} \leq s_i^{S'}$ for $i \in [n]$.

These calculations show that in the following rules we have set up for when placement is legal and how the pair set of the next state is obtained in the strong 2-hook construction that our assumption $p^S \leq s^S$ will hold throughout the construction. 

In the strong 2-hook construction, after placement on index $k$ on $S$ is decided legal, we must either \emph{mark} the index $k$ or leave index $k$ \emph{unmarked}. Whether the placement on index $k$ is legal also depends on whether we intend to mark or unmark the index. 

If the location of index $k$ on $S$ would require us to mark the index after placement, the placement is not legal unless $h^S \leq s_k^S$. In the state $S'$ formed by a legal placement on index $k$ and marking the index, $E^{S'}=E^S$

If instead the location of index $k$ on $S$ would require us to unmark the index after placement, the placement is not legal unless $h^S \leq M_k^S$. And in the state $S'$ formed by a legal placement on index $k$ and marking the index, $E^{S'}=E^S \cup E^{S}_k$.

The rules for whether index $k$ is marked or unmarked after placement are as follows:

\begin{itemize}
\item If $S$ is an unbroken state and $h^S>1$:
\begin{itemize}
\item If $k=x^S_i$, then index $k$ is marked if legal to.
\item If $k>x^S_i$ and $k<n$, then index $k$ is unmarked if legal to.

\end{itemize}

\item If $S$ is a broken state with $h^S>1$: 
\begin{itemize}
\item If $x^S_i \neq x^S_o$ and $k=x^S_i$, then index $k$ is marked if legal to.
\item If $x^S_i \neq x^S_o$ and $k=x_o^S$,  then index $k$ is unmarked if legal to. 
\item If $x^S_i = x^S_o$ and $k=x^S_i$,  then index $k$ is marked if legal to.
\item If $k=y^S_i$ and $k=y^S_i$,  then index $k$ is unmarked if legal to. 
\end{itemize}

\item If $S$ is a state with $h^S=1$, we may place on $n$ and unmark $n$ if legal to. 
\end{itemize}  

The final rules concern how the construction completes. If all possible placements are illegal, we say the 2-hook construction has \emph{ended in failure}. If the strong 2-hook construction ends in a state $S^{(n)}$ where $p^{S^{(n)}}$ is not a 2-hook permutation, we also consider the construction to have ended in failure. However, if $p^{S^{(n)}}$ is justified by $E^{S^{(n)}}$, we say the strong 2-hook construction \emph{finds $p^{S^{(n)}}$} and that $E^{S^{(n)}}$ is its \emph{placement solution}. Regardless of either outcome, we still say the strong 2-hook construction \emph{completes} if $S^{(n)}$ is reached.

We now have all of the rules for the strong 2-hook construction. Recall the example \ref{example} we used for the action of the weak 2-hook construction. Let's try to make the same placements as before but using the strong 2-hook construction rules. 

The strong 2-hook connection starts basically the same as the weak 2-hook connection at state $S^{(0)}$ with $p^{S^{(0)}}=((7,0,0,0,0,0,0)$ and with $E^{S^{(0)}}= \emptyset$. We wish to know if the next placement of the example, which was of value $6$ on index $5$, would still be legal. Note that this placement would obtain a broken state from a broken state, the rules state we must check the legality for placement and then unmarking the index. So we must compute $M_5^{S^{(0)}}$. 

We first compute $s^{S^{(0)}}-p^{S^{(0)}}=(1,6,5,4,3,2,0)$ to find $I(s^{S^{(0)}}-p^{S^{(0)}})$. This lets us see that $E_5^{S^{(0)}}=\{(1,5),(2,5),(3,5)\}$, allowing us to see that $M_5^{S^{(0)}}=6$, so  placement on of value $6$ index $5$ is allowable. We now obtain $S^{(1)}$ by placement on index $5$ of $S^{(0)}$ and then unmarking the index:
\begin{align*}
p^{S^{(1)}} &=(7,0,0,0,6,0,0) \\
E^{S^{(1)}} &= \{(1,5),(2,5),(3,5)\} 
\end{align*}

The next placement was of value $5$ on index $4$ of $p^{S^{(1)}}$. We will be placing on $x_o^{S^{(1)}}$, so according to the rules this index will be unmarked if placed on, so we again need to calculate $M^{S^{(1)}}_4$. As before, we first check that $s^{S^{(1)}}-p^{S^{(1)}}=(0,5,4,4,0,2,0)$, which gives us that $E_4^{S^{(1)}}=\{(2,4)\}$ and  therefore that $M^{S^{(1)}}_4=5$. A placement of value $5$ on index $4$ of $S^{(1)}$ is therefore legal. Making the placement on $S^{(1)}$ and unmarking the index gives us state $S^{(1)}$:
\begin{align*}
p^{S^{(2)}}&=(7,0,0,5,6,0,0) \\ 
E^{S^{(2)}} &= \{(1,5),(2,5),(3,5),(2,4)\}
\end{align*}

If we repeat these basic calculations for every step of example \ref{example} , we will find that each step was allowable and eventually obtain state $S^{(7)}$ with $p^{S^{(7)}}=(7,4,2,5,6,3,1)$ and $E^{S^{(7)}}$ being the justifying pair set from before. So by performing the weak 2-hook construction, we never actually encountered an illegal placement in the strong 2-hook construction, and the final state we reached was a 2-hook permutation with a justifying pair set. Our strengthened version of Lemma \ref{mainprop} shows that the strong 2-hook construction always functions this way.

\begin{theorem}\label{tmain}
\begin{itemize}
\item Let $S$ be a state on $n$ from the strong 2-hook construction. If placement on index $k$ of $p^S$ is legal by the weak 2-hook construction rules, then it is legal by the strong 2-hook construction rules. Therefore the strong 2-hook construction never ends in failure.
\item If $S$ is a final state of the strong 2-hook construction, then $p^S$ is justified by $E^S$ and therefore $p^S \in \mathcal{N}_{\lambda^n,E_n}$
\item Additionally, an integer $n$-tuple $t$ is a 2-hook permutation only if $t$ can be found by the 2-hook construction. 
\end{itemize}
\end{theorem}

Theorem \ref{tmain} is proven by induction on $n$. In order for induction to work, the connection between states on $n$ and states on $n-1$ needs to be understood. Consider that we just found that the state $S=((7,0,0,5,6,0,0), \{(1,5),(2,5),(3,5),(2,4)\})$ can be obtained from the strong 2-hook construction. However, starting with the initial construction for $n=6$ and then following the rules for placement on index 4 results in the state $S'=((6,0,0,5,0,0), \{(1,4),(2,4)\})$. Notice that $h^S=h^{S'}$ and $|[6]-I(p^S)|=|[5]-I(p^{S'})|$. More strongly, if we set $\alpha: [6] - I(p^S) \to [5]-I(p^{S'})$ to be the order preserving bijection between the two sets, one can note that $\alpha(x^S_i) = x^{S'}_i$, $\alpha(x^S_o) = x^{S'}_o$, and $\alpha(y^S_i) = y^{S'}_i$. We even have the striking condition that  $s^S$ restricted to $[6]-I(p^S)$ is  $ (4,4,2)$ and $s^{S'}$ restricted to $[5]-I(p^{S'})$ is $(4,4,2)$. Similarly,  $M^S$ restricted to $[6]-I(p^S)$ is  $ (4,4,4)$ and $M^{S'}$ restricted to $[5]-I(p^{S'})$ is $(4,4,4)$.

The paragraph above is describing a condition that occurs regularly between states of different lengths called the \emph{$\alpha$-correspondence}. To formally define $\alpha$-correspondence for states, we need to first define a different notion of $\alpha$-correspondence that applies to vectors and pair sets. Suppose we have $t^a \in \mathbb{Z}^{n^a}$, $t^b \in \mathbb{Z}^{n^b}$ for $n^b < n^a$.  If there is a bijective function $\alpha: A \to B$ for two sets $A \subseteq [n^a]$, $B \subseteq [n^b]$, $|A|=|B|$, then we say \emph{$t^a$ and $t^b$ are in $\alpha$-correspondence} if $t^a_i=t^b_{\alpha(i)}$ for all $i \in A$. Similarly if $E^a \subseteq E^{n^a}$ and $E^b \subseteq E^{n^b}$, then \emph{$E^a$ and $E^b$ are in $\alpha$-correspondence} when $(i,j) \in E^a$ if and only if $(\alpha(i),\alpha(j)) \in E^b$.

Two states $S^a$ on $[n^a]$ and $S^b$ on $[n^b]$ are in \emph{$\alpha$-correspondence for states} if $|[n^a-1]-I(p^{S^a})|=|[n^a-1]-I(p^{S^a})|$ and there is a bijection $\alpha: [n^a-1]-I(p^{S^a}) \to [n^b-1]-I(p^{S^b})$ for which all the following hold: $\alpha(x_i^{S^a})=x_i^{S^b}$, $\alpha(x_o^{S^a})=x_o^{S^b}$, $\alpha(y_i^{S^a})=y_i^{S^b}$, $s^{S^a}$ is in $\alpha$-correspondence with $s^{S^b} $, and $M^{S^a}$ is in $\alpha$-correspondence with $ M^{S^b}$. Note that $|[n^a-1]-I(p^{S^a})|=|[n^a-1]-I(p^{S^a})|$ means that $h^{S^a}=h^{S^b}$ as a consequence.

 $\alpha$-correspondence between states is a strong condition that has the following consequences for the next placement on each state:

\begin{lemma}\label{linduct} 
If state $S^{(a,0)}$ on $n_a$ and state $S^{(b,0)}$ on $n_b$ are in $\alpha$-correspondence, then the following hold: 

\begin{enumerate}

\item\label{istate} 
 A placement on $i$ in $S_a$ is legal if and only if a placement on $\alpha(i)$ in $S_b$ is legal. 

\item\label{imark}
Index $i$ is marked in $S^{(a,0)}$ if and only if index $\alpha(i)$ is marked in $^{(a,0)}$.

\item\label{inewalpha} If $S^{(a,1)}$ is obtained from a legal placement on $i$ from $S^{(a,0)}$, and $S^{(b,1)}$ is obtained from a legal placement on $\alpha(i)$ from $S^{(b,1)}$, then $S^{(a,1)}$ and $S^{(b,1)}$ are in $\alpha'$-correspondence where $\alpha'$ is the restriction of $\alpha$ to $[n-1]-I(p^{S^{(a,0)}})$. 

\item\label{iedge} 
If $S^{(a,1)}$ is obtained from a sequence of legal placements on $k_1, k_2, \dots k_l$ starting from $S^{(a,0)}$, and $S^{(b,1)}$ is obtained from a sequence of legal placements on $\alpha(k_1),\alpha(k_2), \dots \alpha(k_l)$ starting from $S^{(b,0)}$, and if $s^{S^{(a,0)}}_i=0$ for $i \in I(p^{S^{(a,0)}})$ and $s^{S^{(b,0)}}_i=0$ for $i \in I(p^{S^{(b,0)}})$, then $E^{S^{(a,1)}}-E^{S^{(a,0)}}$ is in $\alpha$-correspondence with $E^{S^{(b,1)}}-E^{S^{(b,0)}}$. 

\end{enumerate}
\end{lemma}
\begin{proof}
Lemma \ref{linduct}.\ref{imark} follows from how $\alpha(x^{S^{(a,0)}}_i)=x^{S^{(b,0)}}_i$, $\alpha(x^{S^{(a,0)}}_o)=x^{S^{(b,0)}}_o$, $\alpha(y^{S^{(a,0)}}_i)=y^{S^{(b,0)}}_i$.

The proof of Lemma \ref{linduct}.\ref{istate} follows from the definition of $\alpha$-correspondence and from Lemma \ref{linduct}.\ref{imark}. The proof of Lemma \ref{linduct}.\ref{inewalpha} follows from Lemma \ref{linduct}.\ref{istate}. 

The proof of Lemma \ref{linduct}.\ref{iedge}  requires some explanation: Suppose in the stated conditions $l=1$. By the rules of the strong 2-hook construction, if $s^{S^{(a,0)}}_i=0$ for $i \in I(p^{S^{(a,0)}})$ then the pairs of $E^{S^{(a,1)}}-E^{S^{(a,0)}}$ must start and end in $[n]-I(p^{S^{(a,0)}})$ and the pairs of $E^{S^{(b,1)}}-E^{S^{(b,0)}}$ must start and end in $[n]-I(p^{S^{(b,0)}})$. 

From this we can reason that if $k$ is legal placement for $S^{(a,0)}$, then $E_k^{S^{(a,0)}}$ is in $\alpha$ correspondence with $E_{\alpha(k)}^{S^{(b,0)}}$. Additionally, by Lemma \ref{linduct}.\ref{imark} the index $k$ is unmarked and  $E_k^{S^{(a,0)}}$ is contained in the pair set of the state  $S^{(a,1)}$ iff  the index $\alpha(k)$ is unmarked and  $E_{\alpha(k)}^{S^{(b,0)}}$ is contained in the pair set of $S^{(b,1)}$.

So $E^{S^{(a,1)}}-E^{S^{(a,0)}}$ is always in $\alpha$-correspondence with $E^{S^{(b,1)}}-E^{S^{(b,0)}}$. If $l>1$ we may apply this reasoning for each pair of placements $k_i$ and $\alpha(k_i)$ for $i\in [l]$. 
\end{proof}

Because of this lemma, $\alpha$-correspondence can lend insight into whether an instance of the strong 2-hook construction completes.

\begin{lemma}\label{lcompletion}
Let $S^a$ on $n^a$ and $S^b$ on $n^b$ be two states in $\alpha$ correspondence and with $s^{S^a}_i=0$ for $i \in I(p^{S^a})$ and $s^{S^b}_i=0$ for $i \in I(p^{S^b})$. Then the sequence $k_1, k_2, \dots k_l$  are legal placements that complete the strong $2$-hook construction when starting from $S_a$ iff the sequence $\alpha(k_1), \alpha(k_2), \dots \alpha(k_l)$  are legal placements that complete the strong $2$-hook construction when starting from $S_b$. 

Additionally if $(p^a, E^a)$ is the last state of the 2-hook construction starting from $S^a$ and if $(p^b, E^b)$ is the last state of the 2-hook construction starting from $S^b$ then $E^a$ justifies $p^a$ iff $E^b$ justifies $p^b$. 
\end{lemma}

\begin{proof}

If $l=1$ then the first statement is the same as Lemma \ref{linduct}.\ref{istate}. If placements $k_1, \dots, k_{i-1}$ have been legal placements on $S^a$ and $\alpha(k_1), \dots, \alpha(k_{i-1})$ have been legal placements on $S^b$, then the next placement $k_i$ is legal for the former resulting state iff the next placement $\alpha(k_i)$ is legal for the latter resulting state by Lemma \ref{linduct}.\ref{istate}. It then follows from this simple induction that the former strong 2-hook construction will complete iff the latter strong 2-hook construction will complete.

The pair sets $E^a-E^{S^a}$ and $E^b-E^{S^b}$ are also in $\alpha$ correspondence by Lemma \ref{linduct}.\ref{iedge}.

So $s_i^{S^a}=0$ for $i \in I(P^{S^a})$ if and only if $s_i^{S^b}=0$ for $i \in I(P^{S^b})$. Since we gave the condition that $s_i^{S^a}=0$ for $i \in [n]-I(P^{S^a})$ and $s_i^{S^b}=0$ for $i \in [n]-I(P^{S^b})$, it follows that $E^a$ justifies $p^a$ if and only if $E^b$ justifies $p^b$. 
\end{proof}

We are now ready to prove that the strong 2-hook construction works.

\begin{proof}[\ref{tmain}]
We will show by induction on $n$ that the strong 2-hook construction finds a $t \in \mathcal{N}_{\lambda^n,E_n}$ and finds a pair set that justifies $t$, and that if $t$ cannot be found by the 2-hook construction then $t$ is not a 2-hook permutation. The base case of $n=1$ is trivial, and we may assume that the inductive hypothesis holds for $n-1$. 

Consider an integer vector $t \in \mathbb{Z}^n$ for which $t_1, \dots t_n$ form a permutation of $[n]$ with $t_1=n$ and $t_n=1$. Let $k_1 \in [n]$ be the index for which $t_{k_1}=n-1$ and $k_2 \in [n]$ be the index for which $t_{k_2}=n-2$. We will assume there is a pair set $E \subseteq E_n$ that justifies $t$.

We first show that if the indices $k_1, k_2$ could not have been the first two placements from the strong 2-hook construction on the initial state $S^{(0)}$ on $n$, then $E$ cannot justify $t$. If $k_1=2$ there is nothing to consider, as $k_1=2$ is a legal placement and $k_2$ can be any other index from 3 to $n-1$. Assume $k_1>2$. We note that $\lambda^n_{k_1}+\omega^n_k=n+1-k_1$. $E^{S^{(0)}}$ is empty, so $M^{S^{(0)}}_{k_1}=n-1$, but for $E \subseteq E_n$ to justify a $t$ with $t_k=n-1$ we must have $E$ include all $k_1-2$ pairs in $E$ that end at $k_1$. Therefore we must have $\{(1,k_1),(2,k_1), \dots,(k_1-2,k_1)\} \subseteq E$.  However if $t_{k_2}=n-2$ for some $k_2 \in \{3,\dots,k_1-2\}$, then $t_{k_2}-(\lambda^n_{k_2}+\omega^n_{k_2})-1=k_2-2$, and so there is no pair set that could justify $t$. Similarly, if $t_{k_2}=n-2$ for $k_2>k_1+1$ then  $t_{k_2}-(\lambda^n_{k_2}+\omega^n_{k_2})=k_2-3$, and there are now only $k_2-4$ edges that could be in $E$ that end at $k_2$. So if $t_{k_2}=n-2$, then $k_2 \in \{2, k_1-1,k_1+1\}$. Similar computations show that if $k_2=k_1-1$ or $k_1+1$ then $k_2-3$ edges must end at $k_2$ so $\{(2,k_2),(3,k_2),\dots,(k_2-2,j)\} \subseteq E$. No edges need to be added if $k_2=2$. Therefore $k_1$ and $k_2$ must be legal as the first two placements of the strong 2-hook construction. 

We now show that the first two placements of the 2-hook construction on the initial state $S^{(0)}$ on $n$ are always legal, and result in a state where can apply Lemma \ref{lcompletion}.  Since $s_2^{S^{(0)}}=n-1$, placement on $2$ on $S^{(0)}$ is legal. Let $S^{(1a)}$ the result of placement on $2$ on $S^{(0)}$. Then $s^{S^{(1a)}}_3=n-2$, so placement on $3$ in $S^{(1a)}$ is legal.  Since $s^{S^{(1a)}}_2=0$, $M^{S^{(1a)}}_k=n-2$ for $k \in \{4,\dots, n-1\}$, so placement on those indices is also legal. Basic calculations using our rules shows that for all $k \in \{3,\dots,n-1\}$, placement on $k$ from $S^{(1a)}$ results in a state $S^{(2a)}$ with $s^{S^{(2a)}}_2=0$ and $s^{S^{(2a)}}_k=0$

Alternatively, since $M^{S^{(0)}}_k=n-1$ for $k \in \{3,\dots,n-1\}$, placement on one of the indices $3,\dots,n-1$ is legal in $S^{(0)}$. Let $S^{(1b)}$ be a state resulting from placement on one of those indices in $S^{(0)}$. Then $s_2^{S^{(1b)}}=n-2$, so placement on $2$ is legal. $s_1^{S^{(1b)}}=0$ and $s_k^{S^{(1b)}}=n-2$, so  $M^{S^{(1b)}}_{k-1}=n-2$, and $M^{S^{(1b)}}_{k+1}=n-2$. Therefore placement on indices $k-1$, $k+1$ is legal as well, and for $j \in \{2,k-1,k+1\}$ placement on $j$ on $S^{(1b)}$ results in a state $S^{(2b)}$ with $s^{S^{(2b)}}_k=0$ and $s^{S^{(2b)}}_j=0$.

Let $S^{[0]}$ be the initial state on $n-1$. We now show that $S^{(1a)}$ and any state from a placement on $S^{(1b)}$ has an $\alpha$-correspondence with a state on $n-1$. The possibilities are as follows:

\begin{itemize}
\item Let $\alpha:[n]-\{2\} \to [n-1]$ be defined as $\alpha(1)=1$, $\alpha(i)=i-1$ for $i>2$. Then $S^{(1a)}$ and $S^{[0]}$ are in $\alpha$-correspondence.

\item Let $S^{(2a)}$ be the state from placing $n-2$ on $2$ in $S^{(1b)}$. Let $S^{[1a]}$ be the state from placing $n-2$ on $k-1$ in $S^{[0]}$. Define the bijection $\alpha: [n]-\{2,k\} \to [n-1]-\{k-1\}$  with $\alpha(i)=i-1$ for $i \in [n]-\{2,k\}$. Then there is in $\alpha$ correspondence from $S^{(2a)}$ to $S^{[1a]}$. 

\item Let $S^{(2b)}$ be the state from placing $n-2$ on $k-1$ in $S^{(1b)}$. Let $S^{[1b]}$ be the state from placing $n-2$ on $k-1$ in $S^{[0]}$.  Define the bijection $\alpha: [n]-\{k-1,k\} \to [n-1]-\{k-1\}$  with $\alpha(i)=i$ for $i \in \{2,\dots,k-2\}$ and $\alpha(i)=i-1$ for $i>k$. Then there is in $\alpha$ correspondence from $S^{(2b)}$ to $S^{[1b]}$. 

\item Let $S^{(2c)}$ be the state from placing $n-2$ on $k+1$ in $S^{(1b)}$. Reconsider state $S^{[1a]}$. Define the bijection $\alpha: [n]-\{k,k+1\} \to [n-1]-\{k\}$  with $\alpha(i)=i$ for $i \in \{2,\dots,k-1\}$ and $\alpha(i)=i-1$ for $i>k+1$. Then there is in $\alpha$ correspondence from $S^{(2c)}$ to $S^{[1a]}$. 
\end{itemize}

Each of these possibilities result in a state $S$ on $n$ in $\alpha$-correspondence with a state $S'$ on $n-1$, where $s^S_i=0$ for $i \in I(p^S)$ and $s^{S'}_i=0$ for $i \in I(p^{S'})$. By Lemma's \ref{linduct} and \ref{lcompletion}, we see that since the inductive hypothesis applies to $S'$, the strong 2-hook construction encounters no illegal placements and can always complete from $S$ and return a $2$-hook permutation with a justifying pair set. 

We now show no other 2-hook permutations exist except for those found by the 2-hook construction. Suppose there is a $t \in \mathcal{N}_{\lambda^n, E_n}$ with justifying pair set $E$ so that $t$ cannot be found by the strong 2-hook construction. 

Let $t_{k_1}=n-1$ and $t_{k_2}=n-2$ as before. By placing value $n-1$ on index $k_1$ and then possibly value $n-2$ on index $k_2$ in the state on the state $S^{(0)}$ on $n$, we can always get one of the states $S^{(1a)}$, $S^{(2a)}$, $S^{(2b)}$ or $S^{(2c)}$. We found above that each of these states is a state $S$ on $n$ that is in $\alpha$-correspondence to some state $S'$ on $n-1$.  We construct $t' \in \mathbb{Z}^{n-1}$ with $t'_{\alpha(i)} = t_i$ for $i \notin [n]-I(p^{S'})$ and will find a justifying pair set for it. 

Let $E^a=\{(i,k_1)|i \in [k_1-2]\} \cup \{(i,k_2)|i \in [k_2-2]\}$.  Then $E^a  \subseteq E$. This means that $E$ cannot contain pairs $(k_1,i) \in E_n$ for $i \in [n]$ because if $(k_1,i) \in E$ then  $\lambda_{k_1}^n+\omega_{k_1}^n+v_{k_1}(E) > \lambda_{k_1}^n+\omega_{k_1}^n+v_{k_1}(E^S)=n-1$. $E$ cannot contain pairs $(k_2,i) \in E_n$ for $i \in [n]$ by similar reasoning. Therefore if $(i,i') \in E-E^a$, then $i, i' \in [n]-\{k_1, k_2\}$. 

Therefore $t'$ is justified by $E^{S'} \cup E'$, where $E'$ is in $\alpha$ correspondence with $E-E^a$. By inductive hypothesis, $t'$ must be able to be found by the strong 2-hook construction by a sequence of placements $k_1,\dots, k_l$ on $S'$. Therefore $t$ must be able to be found by the series of placements $\alpha^{-1}(k_1), \dots, \alpha^{-1}(k_l)$ on $S$, contradicting our assumptions. 
\end{proof}

\section{The Main Result}

So for $2$-hooks, the strong 2-hook construction can find a justifying pair set for every 2-hook permutation. To find the coefficient $c_t$ for an 2-hook permutation $t$, we need to know all possible justifying pair sets for $t$ that contain only pairs in $E_n-E(G)$, which might not even contain the placement solution.  We will rectify this in this section.

Let $\mathcal{E}(G)$ be the power set of $E_n-E(G)$. Let $\mathcal{E}_t$ be the set of all justifying pair sets for $t$. Given arbitrary pair set $E^a \subseteq E_n$, suppose there are indices $a,b,c \in [n]$ with $a<b<c$ so that $b$ is marked, $(b,c) \notin E^a$, $(a,c) \in E^a$, $(a,b),(b,c) \in E_n$. By definition of $b$ being marked $(a,b) \notin E^a$. We call the pair set $E^b=E^a-\{(a,c)\} \cup \{(a,b),(b,c)\}$ the result of \emph{breaking arc $(a,c)$ over $b$} in $E^a$. Similarly, $E^a$ could be said to be the result of \emph{unbreaking arc $(a,c)$} in $E^b$. The important thing to note is that $v(E^a)=v(E^b)$. 

The following lemma is a consequence of the definition of $\alpha$-correspondence for pair setsand definition of unbreaking arcs, but bears statement. 

\begin{lemma}\label{lunbreak}
If $E^a$ and $E^b$ are in $\alpha$ correspondence and if $(i,j), (j,k) \in E^a$, then $(i,j)$ and $(j,k)$ can be unbroken in  $E^a$ iff $(\alpha(i),\alpha(j))$ and $(\alpha(j),\alpha(k))$ can be unbroken in $E^b$. Additionally, unbreaking both pairs of arcs obtains two new pair sets $E^{a'}$ and $E^{b'}$ that are also in $\alpha$-correspondence. 
\end{lemma}

Finding $\mathcal{E}_t$ by breaking arcs of $E$ in every possible sequence sounds somewhat unreasonable, as each pair in $E$  has many places to break. However, it's somewhat more reasonable than that due to the way the placement solution is constructed, and it is why we were marking indices. Let $M$ be the set of all indices marked in the process of finding $t$ through the strong 2-hook construction. 

We now show we can construct  $\mathcal{E}_t$ through arc breaks. 

\begin{theorem}\label{tunbreak}
Let $t$ be an element of $\mathcal{N}_{\lambda^n,E_n}$. Let $E$ be the placement solution for $t$. If $E^a \subseteq E_n$ justifies $t$, then $E^a$ can be obtained from  a series of arc breaks on $E$. Additionally, every arc break in the sequence must always be over a vertex in $M$.
\end{theorem}

\begin{proof}
We take cases based on the index $k \in [n]$ for which $t(k)=n-1$. Suppose $k=2$.  Let $\alpha:[n] \to [n-1]$ be defined as $\alpha(1)=1$ and $\alpha(i)=i-1$ if $i>2$. 

Let $t'$ be the integer $n-1$-tuple with $t_1=n-1$ and in $\alpha$ correspondence with $t$. We know from Lemma \ref{linduct} that $t'$ is justified by $E' \subseteq E_{n-1}$, the pair set in $\alpha$ correspondence with $E$. 

 Suppose $(1,3) \notin E^a$. Let $E^b=\{(\alpha(i),\alpha(j))|(i,j) \in E^a\}$. 
We also have shown previously that there are no pairs in any pair set justifying $t$ that start from 2. The pairs $(i,j)$ in $E^a$ with $i,j>1$ correspond with pairs $(\alpha(i), \alpha(j)) \in E_{n-1}$, which have their lengths preserved by the $\alpha$-correspondence. The pair $(1,a)$ with $a >3$ is the single pair in $E^a$ that starts at index 1. so it corresponds with $(\alpha(1), \alpha(a)) = (1,a-1) \in E_{n-1}$ with length at least 2. Therefore $E^b \subset E_{n-1}$ and is in $\alpha$-correspondence with $E^a$ and justifies $t'$. So we may apply the  inductive hypothesis to $E^b$ to show it is the result of unbreaking arcs from $E'$, and then apply Lemma \ref{lunbreak} to show $E^a$ is the result of unbreaking arcs from $E$.

If $(1,3) \in E^a$, then note that $t_3 \leq n-2 = \lambda^n_3 + \omega^n_3$  so there must be some arc $(3,a) \in E^a$. Note also that a justifying pair set for $t$ can only contain one arc that starts from index 1, so the arc $(1,a)$ is not in $E^a$. Therefore, we may unbreak $(1,3)$ and $(3,a)$ in $E^a$ and apply the previous case to the pair that results.

If $t_k=n-1$ for $k>2$, then consider the states $S^{(0)}, S^{(1)}, \dots, S^{(n)}$ that occur in the sequence of the 2-hook construction when it finds $t$. When $n-1$ is placed on $k$ on the initial state $S^{(0)}$ on $n$, the new state $S^{(1)}$ is broken. Suppose that first state in the process to be unbroken again is state $S^{(l)}$ from placing the value $n-l$ on index $j$ on state $S^{(l-1)}$ for some $j \in [k], l \in [n-1]$. Following the rules for where values can be placed when the state is broken, we can see that the values $n-1,n-3,\dots,n-l+1$ were placed on the indices $\{2,3,\dots,l+1\}-\{j\}$ in some order before the state $S^{(l)}$. After the state $S^{(l)}$, we must place $n-l-1\dots,1$ on $l+2,\dots,1$ in some order to complete the 2-hook construction. This lets us know what value $t_i$ can have on the sets $A=\{1,\dots,l+1\}$ and  $B=\{l+2,\dots,n\}$.

The sum  $\sum_{i \in A}  \lambda^n_i+\omega^n_i$ is $\frac{(2n-l)(l+1)}{2}+1$ while the sum $\sum_{i \in A} t_i$ is $\frac{(2n-l)(l+1)}{2}$. Similarly, the sum  $\sum_{i \in B} \lambda^n_i+\omega^n_i$ is $\frac{(n-l)(n-l-1)}{2}-1$ while the sum of $\sum_{i \in B} t_i$ is $\frac{(n-l)(n-l-1)}{2}$. Therefore any pair set justifying $t$ has exactly one arc from an index in $A$ to an index in  $B$, which we will call the \emph{crossing arc}. Let $(a,b)$ for $a \in A$  $b \in B$ be the crossing arc in $E$, and let $(a',b')$ for $a' \in A$, $b' \in B$ for $E^a$. We also define $E^a_A=\{(i,j): (i,j) \in E, i,j \in A\}$, $E^a_B=\{(i,j): (i,j) \in E, i,j \in B\}$. 

Let $k_1, k_2,\dots,k_l$ be the first $l$ indices placed on by the 2-hook construction to obtain state $S^{(l)}$. Let $S_A^{(0)}$ be the initial state on $l+2$. Let $\beta: \mathbb{Z} \to \mathbb{Z}$ be defined as $\beta(i)=i-(n-l-1)$. As we place value $n-i$ on index $k_i$ of state $S^{(i-1)}$ to get state $S_i$, we may place value $\beta(n-i)$ on index $k_i$ of state $S_A^{(i-1)}$ to get state $S_A^{(i)}$, until we reach $S_A^{(l)}$. States $S_A^{(0)}$ and $S^{(0)}$ technically do not have an $\alpha$ correspondence for any $\alpha$ but we have shown that the 2-hook construction always makes legal placements and finds a 2-hook permutation with justifying pair set, and it's trivial that we are following the $2$-hook construction on for the placements made on states $S_A^{(i-1)}, i \in [l]$ in parallel with the placement made on states $S^{(i-1)}, i \in [l]$. The sequence of placements on $S_A^{(0)}$ can also be seen to  have the same status of marked or unmarked as in $S^{(0)}$  and therefore the same pairs will be added to $E^{S_A^{(i-1)}}$ as in  $E^{S^{(i-1)}}$. Once we place the value $1$ on index $l+2$ of $S_A^{(l)}$, the 2-hook construction completes and finds integer $l+1$-tuple $t^A$ with placement solution $E^A$. $E^A-E^{S_A^{(l)}}$ is a single arc $(a,l+2)$. $E^{S_A^{(l)}}$ and $E^{S^{(l)}}$ are in $id:[l+1]\to [l+1]$ correspondence.

Let $S_B^{(0)}$ the initial configuration on $n-l-1$. Let $\alpha:\{a\} \cup \{l+2,\dots,n\}\to[n-l-1]$ be defined as $\alpha(a)=1$ and $\alpha(i)=i-(n-l-1)$ for $i \neq a$. Then $S_B^{(0)}$ is in $\alpha$-correspondence with $S^{(l)}$, and so we may complete $S_B^{(0)}$ into a 2-hook permutation $t^B$ that is in $\alpha$ correspondence with $t$ and whose placement solution $E^B$ is in $\alpha$-correspondence with $E-E^{S^{(l)}}$. 

We define $E^b_A=E^a_A \cup \{a',l+2\}$. We define $E^b_B=\{(\alpha(i),\alpha(j)): (i,j) \in E^a_B  \} \cup \{1,b'\}$. It is clear that $E^b_A$ justifies $t^A$ and $E^b_B$ justifies $E^b_B$. The induction hypothesis then applies to $E^b_A$ and $E^b_B$, and there is a sequence of arc breaks on each pair set that gives us $E_A$ and $E_B$. We now show the arc breaks can be transferred to $E^a$ to become a series of arc unbreaks that gives us $E$ and finish the induction.

Let $(x^A_1,y^A_1),$ $(x^A_2,y^A_2),\dots$ be the sequence of arc pairs we unbreak in $E^b_A$ and let $(x^B_1,y^B_1),$ $(x^B_2,y^B_2),\dots$ be the sequence of arc pairs we unbreak in $E^b_B$. For each $(x^A_i,y^A_i)$, if $y^A_i$ ends in $l+1$, unbreak the arc $x^A_i$ in $E^b_A$ with the crossing arc in $E^a$ to get a new crossing arc. Otherwise unbreak $x^A_i$ in $E^b_A$ with $y^A_i$ in $E^b_A$. Similarly, for each $(x^B_i,y^B_i)$, if $x^B_i$ starts with $1$, unbreak the crossing arc in $E^a$ with the arc that maps through $\alpha$ to $y^B_i$ to get a new crossing arc. Otherwise unbreak the arc that maps through $\alpha$ to $x^B_i$ in $E^b_B$ with the arc that maps through $\alpha$ to $y^B_i$ in $E^b_B$. We have shown this will eventually become $E$, so we are done showing that every possible $E^a$ that justifies $t$ is found through breaking arcs of the placement solution $E$. 

To show we may only consider arc breaks over $M$, note that if an index of $t$ was unmarked, no arc of the placement solution of $t$, $E$, could ever be broken over it. To see this, just consider that if $(a,b)$ crosses over unmarked vertex $c$ then $(a,c)$ must exist already.
\end{proof}

Note that even for hook shapes other than 2-hooks, there is no unique solution that can reach every other justifying pair set by arc unbreaks. For instance, consider the 3-hook shape $\lambda=$ $(3,1,1,1,0,0,0)$. Then there are multiple pair sets that can be added to $(8,5,4,3,1,0)$ to get $(6,3,5,4,2,1)$ include $\{(1,3),(1,4),(2,5),(2,6)\}$, $\{(1,3),(2,4),(1,5),(2,6)\}$, and $\{(1,3),(2,4),$ $(2,5),(1,6)\}$. So the strong 2-hook construction probably has no analog in this case or others.

For the case of the 2-hook partition, we finally can state our combinatorial description of $\overrightarrow{D}_{\lambda^2_n, E(G)}$ by using Lemma \ref{lctformula}, Proposition \ref{mainprop}, and Theorem \ref{tunbreak}.

\begin{theorem}\label{ttruemain}
Given the 2-hook $\lambda^n$ and $G$ a path with edges on $n$ vertices, $\overrightarrow{D}_{\lambda,G}$ evaluates to:

\[
\sum_{t \in \mathcal{N}_{\lambda^n, E_n}} \left( \sum_{E \in \mathcal{E}_t \cap \mathcal{E}(G) } (-1)^{|E|} \right) s_{l(t)}
\]

\end{theorem}

\section{Enumeration}

It might be questioned how many elements there actually are in $\mathcal{N}_{\lambda^n,E_n}$ and how it compares to $(n-2)!$. It turns out they can be counted exactly.

\begin{lemma}
$N_n=a(n-3)$ where $a(n)$ is given by $a(0)=1$,$a(1)=2$, $a(n)=4 a(n-1) - 2 a(n-2)$ (A006012 from OEIS)
\end{lemma}

Let $N_n = |\mathcal{N}_{\lambda^n,E_n}|$. Note $N_3=1$, $N_4=2$. $N_n$ counts how many sequences of placements there are on the initial state $S^{(0)}$ of length $n$. We use this to form a recurrence. 

If we place $n-1$ on index $2$ in $S^{(0)}$, we have already seen that there are $N_{n-1}$ ways to finish the construction into a $2$-hook permutation from that state. Suppose instead that the placement of $n-1$ is on $k \in \{2,\dots,n-1\}$ on $S^{(0)}$, giving us a broken state. Suppose we continue the construction from there and the value $n-l+1$ is the last value placed before we obtain an unbroken state. Then the proof of Theorem \ref{tunbreak} shows that we have placed on indices $2,\dots, l$ and that $k-2$ values were placed on indices $\{3,\dots, k-1\}$. So for the values $n-2,n-3,\dots,n-l+2$, exactly $l-k$ values must have been placements on their state's $y_i$ index and the other $k-2$ of the values must each have been placed on their state's $x_o$ or $x_i$ index. So we have in this particular case $2^{k-2} \binom{l-2}{l-k} N_{n-l}$ ways to finish. If we fix $l$ and sum over $k$ from $3$ to $l$, we end up with $3^{l-2}$. Therefore
\[
N_n = N_{n-1} + \sum_{i \in \{0,\dots, n-1\}} 3^i N_{n-2-i}
\]

Which can be re arranged to the known recurrence above.

\end{document}